\documentclass[11pt]{amsart}
\usepackage{amsfonts}
\usepackage{amsmath}
\usepackage{amsthm}
\usepackage{url}
\usepackage[all]{xy}
\usepackage{latexsym}
\usepackage{amssymb}
\usepackage[cp850]{inputenc}
\usepackage{amsfonts}
\usepackage{graphicx}

\usepackage[abs]{overpic}

\newtheorem{sat}{Theorem}[section]		
\newtheorem{lem}[sat]{Lemma}
			
\newtheorem{prop}[sat]{Proposition}

\newtheorem*{defi*}{Definition}			
\newtheorem*{bei*}{Example}
\newtheorem*{sat*}{Theorem}				
\newtheorem*{kor*}{Corollary}
\newtheorem*{rmk*}{Remark}				
	
\newtheorem*{quest*}{Question}	
\newtheorem{fact}{Fact}	

\let\ssection=\section
\renewcommand{\section}{\setcounter{equation}{0}\ssection}

\newtheorem*{namedtheorem}{\theoremname}
\newcommand{\theoremname}{testing}
\newenvironment{named}[1]{\renewcommand{\theoremname}{#1}\begin{namedtheorem}}{\end{namedtheorem}}

\theoremstyle{remark}
\newtheorem*{bem}{Remark}

\newtheorem*{namedtheoremr}{\theoremnamer}
\newcommand{\theoremnamer}{testing}

			\newcommand{\BH}{\mathbb H}
\newcommand{\BR}{\mathbb R}			
			
\newcommand{\BS}{\mathbb S}			\newcommand{\BZ}{\mathbb Z}
\newcommand{\BF}{\mathbb F}

		\newcommand{\CF}{\mathcal F}
\newcommand{\CG}{\mathcal G}		\newcommand{\CH}{\mathcal H}
		
\newcommand{\CK}{\mathcal K}

\newcommand{\actson}{\curvearrowright}
\newcommand{\D}{\partial}

\DeclareMathOperator{\Id}{Id}		
\DeclareMathOperator{\Isom}{Isom}	
\DeclareMathOperator{\vol}{vol}		

\DeclareMathOperator{\inj}{inj}
\DeclareMathOperator{\diam}{diam}

\DeclareMathOperator{\rel}{rel}

\DeclareMathOperator{\Ker}{Ker}

\newcommand{\comment}[1]{}

\DeclareMathOperator{\core}{core}

\DeclareMathOperator{\Jac}{Jac}

\newcommand{\Area}{\mathrm{Area}}
\newcommand{\norm}[1]{\lVert #1 \rVert}

\begin{document}

\title[]{Free vs. Locally Free Kleinian Groups}
\author{Pekka Pankka}
\address{Department of Mathematics and Statistics, University of Jyv\"askyl\"a}
\email{pekka.pankka@jyu.fi}
\author{Juan Souto}
\address{IRMAR, Universit\'e de Rennes 1}
\email{juan.souto@univ-rennes1.fr}
\date{\today}

\thanks{P.P. was partially supported by the Academy of Finland project 283082.}

\begin{abstract}
We prove that Kleinian groups whose limit sets are Cantor sets of Hausdorff dimension $<1$ are free. On the other hand we construct for any $\epsilon>0$ examples of non-free purely hyperbolic Kleinian groups whose limit set is a Cantor set of Hausdorff dimension $<1+\epsilon$.
\end{abstract}

\maketitle

\section{Introduction}

In this paper we prove the following theorem:

\begin{sat}\label{main-dim}
Let $\Gamma\subset\Isom_+(\BH^n)$ be a discrete, torsion free, and non-elementary group of orientation preserving isometries of $\BH^n$. If the limit set $\Lambda_\Gamma\subset\BS^{n-1}$ has Hausdorff dimension $\dim_\CH(\Lambda_\Gamma)<1$, then $\Gamma$ is a free group.
\end{sat}

Recall that a discrete subgroup of isometries of the hyperbolic space $\BH^n$ is {\em elementary} if it contains an abelian subgroup of finite index. The \emph{limit set} $\Lambda_\Gamma\subset\BS^{n-1}=\D_\infty\BH^n$ of a discrete non-elementary subgroup $\Gamma$ of $\Isom_+(\BH^n)$ is the set of accumulation points of the orbit $\Gamma x$ of some -- and hence any -- point $x\in\BH^n$. We refer to \cite{Kapovich-book,Ratcliffe} for basic facts and definitions on hyperbolic space and its isometries, and to \cite{McMullen,Sullivan} for a discussion of the relation of the Hausdorff dimension $\dim_\CH(\Lambda_\Gamma)$ of the limit set $\Lambda_\Gamma$ to other classical invariants such as the critical exponent $\delta_\Gamma$ of the group $\Gamma$. 

Limit sets are perfect which implies, under the assumptions of Theorem \ref{main-dim}, that $\Lambda_\Gamma$ is a Cantor set. Also, since the critical exponent $\delta_\Gamma$ bounds the Hausdorff dimension $\dim_{\CH}(\Lambda_\Gamma)$ from below, the group $\Gamma$ in Theorem \ref{main-dim} does not contain free abelian groups of higher rank. Taking into account these two facts, we have that Theorem \ref{main-dim} follows, in the finitely generated case, from the following classical result of Kulkarni \cite{Kulkarni} (see also \cite{KK} for a proof):

\begin{sat*}[Kulkarni]
Every finitely generated discrete torsion free subgroup $\Gamma\subset\Isom_+(\BH^n)$ whose limit set $\Lambda_\Gamma$ is a Cantor set is the free product of elementary subgroups $\Gamma=P_1*\dots*P_k$. In particular, if the maximal abelian subgroups of $\Gamma$ are cyclic, $\Gamma$ is free.
\end{sat*}

\begin{bem}
The decomposition provided by Kulkarni's theorem has the property that parabolic subgroups of $\Gamma$ can be conjugated into the factors $P_i$. In particular, $\Gamma$ splits as a free product relative to its parabolic subgroups.
\end{bem}

Since Theorem \ref{main-dim} is, for finitely generated groups, a direct consequence of Kulkarni's theorem, it is perhaps natural to wonder if this stronger result remains true in the generality of Theorem \ref{main-dim}, i.e.\;for possibly {\em infinitely generated} groups. That this is not the case is basically the content of our second theorem:

\begin{sat}\label{grope}
For each $\epsilon>0$ there is a discrete purely hyperbolic subgroup $\Gamma$ of $\Isom_+(\BH^3)$ which is not free but whose limit set $\Lambda_\Gamma$ is a Cantor set with Hausdorff dimension $\dim_\CH(\Lambda_\Gamma)\le 1+\epsilon$.
\end{sat}

\begin{bem}
Note at this point that it follows from Kulkarni's theorem that any discrete torsion free subgroup of $\Isom_+(\BH^n)$ whose maximal abelian groups are cyclic, and whose limit set is a Cantor set, is {\em locally free}, meaning that every finitely generated subgroup is free. This applies in particular to the groups considered in Theorem \ref{main-dim} as well as to those provided by Theorem \ref{grope} and explains the title of this paper.
\end{bem}

It follows from Theorem \ref{grope} that the proof of Theorem \ref{main-dim} cannot be purely topological because it does not suffice to know that the limit set is a Cantor set, or even a tame Cantor set -- every Cantor set in $\BS^2$ is tame. The main geometric ingredient in the proof of Theorem \ref{main-dim} is a result of Besson, Courtois and Gallot \cite{BCG} which implies that, for $\Gamma$ as in the statement, there is a self-map $\BH^n/\Gamma\to\BH^n/\Gamma$ which decreases the 2-dimensional Hausdorff measure of sets in $\BH^n/\Gamma$. The existence of this map has been already used in a similar setting by Kapovich \cite{Kapovich-GAFA}, who showed that non-elementary and torsion free groups $\Gamma\subset \Isom_+(\BH^n)$ with critical exponent $\delta_\Gamma$ have homological dimension $\hbox{hdim}(\Gamma)\le1+\delta_\Gamma$. In our setting this implies that the group $\Gamma$ in Theorem \ref{main-dim} has homological dimension 1. This is however not enough to prove that $\Gamma$ in Theorem \ref{main-dim} is free because all locally free groups, for example those constructed to prove Theorem \ref{grope}, have homological dimension 1. 

Theorem \ref{main-dim} follows easily by constructing a free basis for $\Gamma$ once we establish the following fact: {\em If $H\subset G\subset\Gamma$ are finitely generated with $G$ filled by $H$, then $G$ is generated by elements in $\Gamma$ which translate a fixed base point in $\BH^n$ by less than some constant depending only on $H$ and $\Gamma$} (see Proposition \ref{diam-bound} for the precise statement). This will be proved basically as follows. 

We use the Besson-Courtois-Gallot map to find a 2-dimensional complex $X_G\subset\BH^n/G$ with $\pi_1(X_G)=G$, whose area is bounded from above independently of $G$ and which solves a certain Dirichlet problem. Together with the assumption that $H$ fills $G$, this implies that balls in $X_G$ centered at points in the thick part of $\BH^n/G$ contain definite amount of area. At this point we can use the standard strategy to show that a global area upper bound together with a lower bound for the area of small balls implies an upper bound on the diameter, and hence on the length of those elements needed to generate the fundamental group $\pi_1(X_G)=G$. 

This paper is organized as follows. In Section \ref{sec:AC} we discuss an algebraic condition ensuring the existence of a free basis for a locally free group. In Section \ref{sec-proof-thm} we transform this algebraic condition to Proposition \ref{diam-bound} mentioned above. Then, in Section \ref{sec:BCG}, we discuss the Besson-Courtois-Gallot map and the existence of complexes $X_G$. The proof of Proposition \ref{diam-bound} is then finished in Section \ref{sec:diam}. The example proving Theorem \ref{grope} is given in the final section.

\medskip

\noindent{\bf Acknowledgements.} This paper was written during several visits to Rennes and Jyv\"askyl\"a. Both authors thank each others institutions for hospitality. 


\section{Algebraic considerations}
\label{sec:AC}

In this section we discuss the algebraic components of the proof of Theorem \ref{main-dim}. Our main tool is some basic Bass--Serre theory and we assume a certain familiarity with such concepts as the {\em fundamental group $\pi_1(\CG)$ of a graph of groups $\CG$}, the relation between representations of a given group $G$ as the fundamental group of graphs of groups, and actions of $G$ on simplicial trees. We just recall some terminology. A group $G$ {\em splits as a free product} if it is the fundamental group of an essential graph of groups where one of the edge groups is trivial. Here essential means that there is no vertex of degree $1$ for which the inclusion of the edge group into the vertex group is an isomorphism. Equivalently, the associated Bass--Serre tree does not have vertices of valence $1$. A free splitting of a group $G$ is \emph{a free splitting relative to a collection $\{H_i\}$ of subgroups of $G$} if each $H_i$ is conjugated to a subgroup in one of the vertex groups. Equivalently, $G$ splits as a free product relative to $\{H_i\}$ if and only if there is a simplicial tree $T$ without vertices of valence $1$ and a minimal action $G\actson T$ such that some edge stablizer is trivial and such that each $H_i$ has a global fixed point in $T$. We will be mostly interested in groups $G$ which do not split as a free product relative to a subgroup $H$; in this case we will say that $H$ {\em fills} $G$. We refer to \cite{Serre-book} for details on the Bass--Serre theory.

\subsection{Detecting free groups}

As mentioned in the introduction, we will be mostly interested in {\em locally free groups}, meaning that all finitely generated subgroups are free. In this section we discuss a criterium ensuring that certain locally free groups are actually free.

\begin{prop}\label{criterium}
Let $\Gamma$ be a countable locally free group and suppose that for every finitely generated subgroup $H\subset\Gamma$ the following condition is satisfied:
\begin{itemize}
\item[(*)] Any increasing sequence $G_1\subset G_2\subset G_3\subset \cdots$ of finitely generated subgroups of $\Gamma$ which are filled by $H$ stabilizes.
\end{itemize}
Then $\Gamma$ is free.
\end{prop}

The first step of the proof of Proposition \ref{criterium} is the following well-known fact, which we prove for the convenience of the reader.

\begin{lem}\label{core-exists}
If $\BF$ is a finitely generated free group and $A\subset\BF$ is a subgroup then there is a unique free factor $\core_A(\BF)$ of $\BF$ with $A\subset\core_A(\BF)$ and such that $A$ fills $\core_A(\BF)$.
\end{lem}
\begin{proof}
We order the set of free factors of the finitely generated free group $\BF$ by inclusion and note that any descending sequence stabilizes by Grushko's theorem. It follows that any subgroup $A$ fills some free factor $\bar A$ of $\BF$. Suppose that $\bar A$ and $\bar A'$ are free factors of $\BF$ filled by $A$ and let $B$ be such that $\BF=\bar A'*B$. The action of $A$ on the Bass-Serre tree associated to $\bar A'*B$ has a global fixed point because $A\subset\bar A'$. Hence so does the action of $\bar A$ because $\bar A$ is filled by $A$. It follows that $\bar A\subset\bar A'$. Arguing symmetrically we obtain the opposite inclusion, proving hence that $\bar A=\bar A'$, as we needed to show.
\end{proof}

Supposing now that $\Gamma$ is locally free, and fixing a finitely generated subgroup $H\subset\Gamma$, let $\CF(H)$ be the set of finitely generated subgroups of $\Gamma$ which contain $H$; note that each element of $\CF(H)$ is a finitely generated free group. We order $\CF(H)$ by inclusion and note that $\CF(H)$ is a directed set, meaning that any two elements have a common upper bound -- take the subgroup generated by both subgroups. Let also $\CK(H)$ be the subset of $\CF(H)$ consisting of finitely generated subgroups of $\Gamma$ filled by $H$. Note that $H\in\CK(H)$, meaning that $\CK(H)\neq\emptyset$.

We consider the map 
\begin{equation}\label{eq-core-map}
\core_H:\CF(H)\to\CK(H),\ \ G\mapsto\core_H(G),
\end{equation}
where $\core_H(G)$ is as in Lemma \ref{core-exists}. The argument used to prove Lemma \ref{core-exists} shows that the map $\core_H(\cdot)$ is order preserving, that is,
\begin{equation}\label{core-order}
G\subset K\Rightarrow \core_H(G)\subset\core_H(K).
\end{equation}
Moreover, the very definitions of $\core_H(G)$ and $\CK(H)$ imply  
\begin{equation}\label{core-fix1}
\core_H(G)=G\ \ \hbox{for all}\ G\in\CK(H).
\end{equation}
In particular,
\begin{equation}\label{core-fix}
\core_H(\core_H(G))=\core_H(G)
\end{equation}
for every $G\in\CF(H)$. 

Having these observations at our disposal, we are now ready to prove Proposition \ref{criterium}.

\begin{proof}[Proof of Proposition \ref{criterium}]
Since the set $\CF(H)$ of all finitely generated subgroups of $\Gamma$ containing $H$ is a directed system, it follows from (*) and equations \eqref{core-order} and \eqref{core-fix} that the set $\CK(H)$, the image of the map $\core_H(\cdot)$, has a unique maximal element which we will denote from now on by $\core(H)$. Note that, we have $\core(H)=\core_H(G)$ for each finitely generated subgroup $G\subset\Gamma$ containing $\core(H)$. In particular, $\core(H)$ is a free factor of any finitely generated subgroup $G$ of $\Gamma$ with $\core(H)\subset G$.

To prove Proposition \ref{criterium} we will construct inductively a certain a filtration 
$$G_1\subset G_2\subset G_3\subset\cdots$$
of $\Gamma$. Enumerate the elements of $\Gamma$ as $g_0=\Id,g_1,g_2,\dots$ and set $G_0=\Id$. Supposing that we have already constructed $G_i$, set
$$G_{i+1}=\core(\langle G_i,g_{i+1}\rangle).$$
Since $G_{i+1}$ contains $G_i=\core(\langle G_{i-1},g_i\rangle)$, it follows that $G_i$ is a free factor of $G_{i+1}$ for all $i$. We can hence construct a free basis of the group $\cup_{i=1}^\infty G_i$ as follows: take a free basis of $G_1$, extend it to a free basis of $G_2$ using that $G_1$ is a free factor of $G_2$, extend it to a free basis of $G_3$ using that $G_2$ is a free factor of $G_3$, etc... The existence of this free basis proves that $\cup_{i=1}^\infty G_i$ is a free subgroup of $\Gamma$. However, since $g_i\in G_i$ for all $i$, we have $\Gamma=\cup G_i$. This concludes the proof of Proposition \ref{criterium}.
\end{proof}

Before moving on observe that to obtain the desired conclusion in Proposition \ref{criterium} it suffices to prove that (*) holds for subgroups $H$ which contain a non-abelian free group.

\subsection{A topological lemma}
We will be constantly working with simplicial complexes -- we recall some terminology. Following \cite{Eells-Fuglede-book} we will say that a simplicial complex $X$ is {\em admissible} if (1) $X$ equals the closure of the maximal dimensional sets and (2) the codimension 2 skeleton does not locally separate $X$. Unless we mention it explicitly, all complexes will be assumed to be admissible.

Under a subcomplex $Z$ of a simplicial complex $X$ we understand a closed subset which with respect to some simplicial structure on $X$ is a union of simplices. The interior and boundary of a subcomplex are its interior and boundary in the usual sense of the pointset topology. We denote the boundary of a subcomplex $Z$ by $\D Z$ and say that it is {\em collared} if $X$ contains a subcomplex $U$ whose interior contains $\D Z$ and such that the inclusion $\D Z\hookrightarrow U$ admits a left-inverse $U\to\D Z$ with respect to which $U$ becomes an interval bundle over $\D Z$. If $Z\subset X$ is a subcomplex and $*\in Z$ is a base point, we denote by $[\pi_1(Z,*)\to\pi_1(X,*)]$ the image of the homomorphism $\pi_1(Z,*)\to\pi_1(X,*)$ induced by the inclusion. As so often, we will drop the reference to the base point, hoping that no confusion arises. In particular, $[\pi_1(Z)\to\pi_1(X)]=\Id$ just means that every loop in $Z$ is homotopically trivial in $X$.

\begin{lem}\label{BS-1}
Let $X$ be a connected 2-dimensional simplicial complex with dense interior, $Y\subset X$ a connected subcomplex with $[\pi_1(Y)\to\pi_1(X)]\neq\Id$, and suppose that $\pi_1(X)$ does not split relative to $[\pi_1(Y)\to\pi_1(X)]$. Suppose in addition that $Z\subset X\setminus Y$ is a connected subcomplex with dense interior, with collared boundary, and with 
\[
[\pi_1(Z)\to\pi_1(X)]=\Id.
\]
Then there is a connected subcomplex $\hat Z\subset X$ with dense interior, with $Y\cap\hat Z=\emptyset$, $Z\subset\hat Z$, 
\[
[\pi_1(\hat Z)\to\pi_1(X)]=\Id,
\]
and such that $\D\hat Z$ is a connected component of $\D Z$.
\end{lem}

\begin{proof}
Let $U$ be the connected component of $X\setminus Z$ containing $Y$ and let $\hat Z$ be its complement. The subcomplex $\hat Z$ has dense interior, contains $Z$, is disjoint of $Y$ and its boundary is contained in $\D Z$. Moreover, Seifert--van Kampen yields a graph of groups decomposition of $\pi_1(X)$ with vertex groups $[\pi_1(U)\to\pi_1(X)]$ and $[\pi_1(\hat Z)\to\pi_1(X)]$ and whose edge groups are the groups $[\pi_1(\Sigma)\to\pi_1(X)]$, where $\Sigma$ ranges over the connected components of $\D\hat Z$. The assumption $[\pi_1(Z)\to\pi_1(X)]=\Id$ implies that the edge groups are trivial. 
Thus we have a graph of groups decomposition with trivial edge groups, with two vertex groups $[\pi_1(U)\to\pi_1(X)]$ and $[\pi_1(\hat Z)\to\pi_1(X)]$, and with an edge connecting those two vertices for each connected component of $\D\hat Z$. However, since $Y\subset U$, it follows from the assumption that $\pi_1(X)$ does not split relative to the non-trivial subgroup $[\pi_1(Y)\to\pi_1(X)]$, that this graph of groups must be trivial, meaning that $[\pi_1(\hat Z)\to\pi_1(X)]=\Id$ and that there is a single edge, meaning that $\D\hat Z$ is connected, as we needed to prove. \end{proof}

\section{Proof of Theorem \ref{main-dim}}\label{sec-proof-thm}
The next two sections are devoted to the proof of the following claim:

\begin{prop}\label{diam-bound}
Let $\Gamma\subset\Isom_+(\BH^n)$ be as in Theorem \ref{main-dim}, $H\subset\Gamma$ a finitely generated non-elementary subgroup and $*\in\BH^n$ a base point. Then there is a constant $d>0$ such that each finitely generated intermediate subgroup $H\subset G\subset\Gamma$ filled by $H$ is generated by elements $g$ satisfying $d_{\BH^n}(*,g *)\le d$.
\end{prop}

\begin{bem}
The chosen base point $*\in\BH^n$ determines, for every torsion free subgroup $G\subset\Isom_+(\BH^n)$, a base point in $M_G=\BH^n/G$ which we still denote by $*$. Using this base point we obtain an identification between $G$ and $\pi_1(M_G,*)$. From now on we assume that each base point in an appearing hyperbolic manifold is the projection of the point $*$ chosen in Proposition \ref{diam-bound}.
\end{bem}

Assuming Proposition \ref{diam-bound}, we prove Theorem \ref{main-dim}:

\begin{named}{Theorem \ref{main-dim}}
Let $\Gamma\subset\Isom_+(\BH^n)$ be a discrete, torsion free, and non-elementary group of orientation preserving isometries of $\BH^n$. If the limit set $\Lambda_\Gamma\subset\BS^{n-1}$ has Hausdorff dimension $\dim_\CH(\Lambda_\Gamma)<1$, then $\Gamma$ is a free group.
\end{named}
\begin{proof}
Let $\Gamma\subset \Isom_+(\BH^n)$ be as in the statement and recall that, as mentioned in the introduction, $\Gamma$ is locally free. To prove that $\Gamma$ is actually free we will use the criterium given in Proposition \ref{criterium}. Suppose thus that $H\subset\Gamma$ is a finitely generated non-elementary subgroup and let 
$$H\subset G_1\subset G_2\subset\cdots$$
be further finitely generated subgroups filled by $H$. By Proposition \ref{diam-bound} we know that each $G_i$ is generated by elements which translate the base point $*$ at most a uniformly bounded amount $d$. Since the discrete group $\Gamma$ contains only finitely many elements with $d_{\BH^n}(*,g *) \le d$ it follows that the sequence stabilizes, as we needed to prove.
\end{proof}


\section{Small area extensions}
\label{sec:BCG}

In this section we suppose the notation of the previous section and we fix an isomorphism between $H$ and the fundamental group of a graph $Y$ -- this is possible because $\Gamma$ is locally free: 
\begin{itemize}
\item[($\star$)] $\Gamma\subset\Isom_+\BH^n$ is as in the statement of the Theorem \ref{main-dim}, $H\subset\Gamma$ a non-elementary finitely generated subgroup, $Y$ a finite connected graph, $*\in Y$ a base point, and $\tau^H:(Y,*)\to(M_\Gamma,*)$ a $\pi_1$-injective piecewise smooth embedding with $\tau^H_*(\pi_1(Y,*))=H$. 
\end{itemize}

With notation as in ($\star$), suppose that $G\subset\Gamma$ is a finitely generated subgroup containing $H$. Under an {\em extension of $\tau^H$ to $G$} we will understand a connected admissible 2-dimensional simplicial complex $X_G$ containing a collared $Y\times [-1,1]\subset X$ copy of $Y=Y\times \{0\}$, together with a $\pi_1$-injective map
$$\tau \colon X_G\to M_\Gamma$$ 
extending $\tau^H$ and satisfying $\tau_*(\pi_1(X_G,*)) = G$. Recall that a complex $X_G$ is admissible if $2$-simplices are dense in $X_G$ and the codimension 2 skeleton does not separate $X_G$ locally. 

\begin{bem}
Each extension $\tau \colon X_G\to M_\Gamma$ of $\tau^H$ to $G$ lifts uniquely to a map $\tau \colon(X_G,*)\to(M_G,*)$, which we will not distinguish from the original map.
\end{bem}

Taking into account that the group $G$ is free it is easy to construct an extension of $\tau^H$. Start by taking a graph $Z$ and a map $\tau^Z:Z\to M_\Gamma$ inducing an isomorphism between $\pi_1(Z)$ and $G$, choose a map $Y\to Z$ representing the injection $H\to G$, let $X_G$ be the 2-complex obtained by gluing $Y\times[-1,1]$ to $Z$ via $Y\times\{1\}=Y \to Z$, and extend the maps $\tau^H$ and $\tau^Z$ to a map $\tau:X_G\to M_\Gamma$ in the correct homotopy class.
\medskip

Our goal is to prove that, under suitable conditions, there are extensions $\tau\colon X_G\to M_G$ of $\tau^H$ to $G$ which have small \emph{area}
\[
\Area(\tau)\stackrel{\tiny\hbox{def}}= \int_{\tau(X_G)} \vert\tau^{-1}(y)\vert \mathrm{d}\CH^2(y),
\]
where $\CH^2$ is the Hausdorff 2-measure in the image $M_G$. Besides having small area, we also want our extension to satisfy a certain convexity property:

\begin{prop}\label{prop:good extension}
With notation as in ($\star$) there exists a constant $C>0$ so that for each finitely generated subgroup $G\subset\Gamma$ containing $H$ there is an extension
$$\tau\colon X_G\to M_G$$
of $\tau^H$ with $\Area(\tau) < C$. Moreover, if $B\subset M_G$ is a convex submanifold and $K\subset X_G\setminus Y$ is a compact subcomplex with dense interior, whose boundary $\D K$ is mapped by $\tau$ into $B$, i.e.\;$\tau(\D K)\subset B$, and such that the restriction $\tau\vert_K$ of $\tau$ to $K$ is homotopic relative to $\D K$ to a map with values in $B$, then actually $\tau(K)\subset B$.
\end{prop}

To prove Proposition \ref{prop:good extension} we use a construction of Besson, Courtois, and Gallot \cite{BCG} to prove that there is an extension with uniformly bounded area. At this point we would like to obtain the desired extension by solving the Plateau problem, but since it is much simpler and actually sufficient for out purposes, we will take $\tau$ to be a solution of the Dirichlet problem with respect to some suitably chosen Riemannian metric on the complex $X_G$.

\subsection{Bounded area extensions}
With notation as in ($\star$) suppose that $G\subset\Gamma$ is a finitely generated group containing $H$. The limit set $\Lambda_G$ of $G$ has Hausdorff dimension bounded from above by that of $\Lambda_\Gamma$:
\[
\dim_\CH(\Lambda_G)\le\dim_\CH(\Lambda_\Gamma)\stackrel{\tiny{\hbox{def}}}=\lambda<1.
\]
Being a finitely generated group whose limit set has Hausdorff dimension less than $1$, $G$ is geometrically finite \cite{KK}. In particular, $G$ has critical exponent 
\[
\delta_G=\dim_\CH(\Lambda_G)
\]
and there is a $G$-invariant conformal density of exponent $\delta_G$, the Patterson-Sullivan measure of $G$ \cite{Sullivan84}. A construction of Besson-Courtois-Gallot \cite{BCG} (see also \cite{SoutoNM}) shows that the existence of such a conformal density implies the existence of a smooth map, the so-called {\em natural map},
$$F\colon M_G \to M_G$$ 
homotopic to the identity and whose $p$-Jacobian is bounded everywhere by
\begin{equation}\label{eq-Jac-bound}
\Jac_pF\le\left(\frac{\delta_G+1}p\right)^p.
\end{equation}
Here the $p$-Jacobian is the function on the $p$-Grassmanian of $M_G$ which associates to each $p$-dimensional subspace $V\subset T_xM$ the absolute value of the determinant of $dF_x\colon V\to dF_x(V)$. More formally, if $V\subset T_xM$ is a subspace of dimension $p$, then $\Jac_pF(V)$ is the norm of the wedge product $dF_x(e_1)\wedge\dots\wedge dF_x(e_p)$, where $e_1,\dots,e_p$ is an orthonormal basis of $V$. Armed with the natural map we prove:

\begin{lem}
\label{lemma:BCG}
There exists a constant $C>0$ so that, for each finitely generated subgroup $G$ of $\Gamma$ containing $H$, there is an extension $\tau\colon X_G \to M_G$ of $\tau^H$ to $G$ with $\Area(\tau)<C$.
\end{lem}

\begin{proof}
We are going to use the natural map $F\colon M_G\to M_G$ to prove that any arbitrary extension $X_G\to M_G$ of $\tau^H$ to $G$ is homotopic to a map with the desired properties. To that end note that \eqref{eq-Jac-bound} yields for $p=2$ the bounds
\[
\Jac_2F\le\left(\frac{\delta_G+1}2\right)^2\le\left(\frac{\lambda+1}2\right)^2\stackrel{\tiny\hbox{def}}=\mu<1,
\]
where, as before, $\lambda=\dim_\CH(\Lambda_\Gamma)$ is the Hausdorff dimension of the limit set of our ambient group $\Gamma$. In particular, this implies, for any map $\tau \colon X_G \to M_\Gamma$ of finite area, that 
$$\Area(F\circ \tau) \le \mu\cdot \Area(\tau).$$
Since $F$ is homotopic to the identity and $Y$ is collared in $X$, we may fix a map $\nu \colon Y\times [-1/2,1/2] \to M_\Gamma$ of finite area so that 
\[
\nu|_{Y\times \{0\}} = \tau^H\ \ \hbox{and}\ \ \nu|_{Y\times \{\pm 1/2\}} = F\circ \tau^H.
\] 

Suppose now that $\tau \colon X_G \to M_G$ is a piecewise smooth extension of $\tau^H$ to $G$. We define $\tau'\colon X_G\to M_G$ by
\[
\tau'|_{Y\times [-1/2,1/2]} = \nu
\]
and
\[
\tau'(y,t) = \left\{ \begin{array}{ll}
(F\circ \tau)(y,2t-1), & \hbox{if}\ (y,t)\in Y \times [1/2,1], \\
(F\circ \tau)(y,2t+1), & \hbox{if}\ (y,t)\in Y \times [-1,-1/2].
\end{array}\right.
\]
The map $\tau'$ has area
\begin{eqnarray*}
\Area(\tau') &=& \Area(F \circ \tau) + \Area(\nu) \\
&\le& \mu\cdot\Area(\tau) + \Area(\nu).
\end{eqnarray*}
In particular, if $\Area(\tau)>2\Area(\nu)/(1-\mu)$, we get that
$$\Area(\tau') \le \frac{1+\mu}{2}\Area(\tau).$$
Thus, by iterating the application of the map $F$ if necessary, we find an extension $\bar\tau \colon X_G \to M_G$ of $\tau^H$ to $G$ for which 
$$\Area(\bar\tau) < 2 \Area(\nu)/(1-\mu)\stackrel{\tiny\hbox{def}}=C,$$
as we needed to prove.
\end{proof}

\subsection{Energy minimizers}
Continuing with the same notation, suppose that we have an admissible $2$-dimensional complex $X_G$ which is moreover endowed with an arbitrary piecewise smooth Riemannian metric $\rho$ in the sense of Eells and Fuglede \cite{Eells-Fuglede-book}. We denote by $\mathcal{F}(X_G;\tau^H)$ the class of continuous Sobolev maps 
$$\tau \colon X_G \to M_G$$ 
in $W^{1,2}(X_G,M_G;\rho)$ which extend $\tau^H$ to $G$.

If $\tau\in\CF(X_G;\tau^H)$ is such an extension and if $x\in X_G$ is a manifold point in which the weak differential $D\tau_x$ of $\tau$ exists, 
let 
\[
\norm{D\tau_x}^2_\rho = \frac{1}{2} \left( |D\tau_xe_1|^2 + |D\tau_xe_2|^2\right)
\]
be the Hilbert--Schmidt norm of the differential $D\tau_x$, where $\{e_1,e_2\}$ is a $\rho$-orthonormal basis of $T_x X$ and the norms on the right are computed using the hyperbolic metric of $M_G$. Note that $\norm{D\tau_x}^2_\rho\ge\Jac_2\tau\vert_x$ with equality if and only if $\tau$ is conformal at $x$.

The {\em 2-energy} or {\em Dirichlet energy} of $\tau\in \mathcal{F}(X_G;\tau^H)$ with respect to $\rho$ is then defined as
\begin{equation}
\label{eq:tau_g}
E(\tau;\rho) = \int_{X_G} \norm{D\tau_x}^2_\rho \vol_\rho.
\end{equation}
If the metric $\rho$ is understood from the context we will write simply $E(\tau)$.

\begin{lem}\label{lem-dirichlet}
With the same notation as above, let $\rho$ be a Riemannian metric on $X_G$. There exists a (unique) minimizer $\tau$ in $\mathcal{F}(X_G;\tau^H)$ for the $2$-energy with respect to $\rho$, that is,
\[
E(\tau;\rho) = \inf_{\tau'\in \mathcal{F}(X_G;\tau^H)} E(\tau';\rho).
\]
\end{lem}

In \cite{Eells-Fuglede-book} Eells and Fuglede present solutions to Dirichlet problems for maps between Riemannian polyhedra in different contexts. We just mention the necessary modifications to obtain Lemma \ref{lem-dirichlet}.

In the proof of Theorem 11.3 in \cite{Eells-Fuglede-book} we may replace the space $C_\psi(X_G,M_\Gamma)$ of continuous maps extending a given map $\psi \colon X_G^{(1)}\to M_\Gamma$, defined on the $1$-skeleton $X_G^{(1)}$, by our space $\mathcal{F}(X_G,M_\Gamma)$. We obtain a H\"older continuous energy minimizer $\tau\in \mathcal{F}(X_G;M_\Gamma)$, which extends $\tau^H$. Furthermore, by the argument used in the proof of Theorem 11.2 in \cite{Eells-Fuglede-book}, we have $\tau_*(\pi_1(X_G,*)) = G$. Note that, although Theorems 11.2 and 11.3 in \cite{Eells-Fuglede-book} are formally stated \emph{a priori} only for compact targets, the existence of minimizers holds also in our setting of non-compact hyperbolic manifold targets because we are anchoring our maps with $\tau^H$; see \cite[p.211]{Eells-Fuglede-book}. This concludes the discussion of Lemma \ref{lem-dirichlet}.\qed

\medskip

We prove that the minimizers provided by Lemma \ref{lem-dirichlet} satisfy the convexity property in Proposition \ref{prop:good extension} -- the point is that projections to convex sets reduce energy. 

\begin{lem}
\label{lemma:facts_on_minimizers}
With the same notation as in Lemma \ref{lem-dirichlet}, let $\tau\colon X_G\to M_G$ be the minimizer in $\mathcal{F}(X_G;\tau^H)$ for the $2$-energy with respect to $\rho$. If $B\subset M_G$ a convex submanifold of co-dimension $0$ and if $K\subset X_G\setminus Y$ is a compact subcomplex and $\tau(\partial K) \subset B$, and if $\tau\vert_K$ is homotopic relative to $\D K$ to a map with values in $K$, then $\tau(K)\subset B$.
\end{lem}

\begin{proof}
Let $\hat M\to M_G$ be the cover of $M_G$ with fundamental group $\pi_1(B)$ and note that $B$ lifts homeomorphically to $\hat M$. Note also that the map $\tau\vert_K\colon K\to M_G$ lifts to a map $\hat\tau\colon K\to\hat M$. Consider the convex-projection $\pi\colon \hat M\to B$, where $\pi$ is the map determined by $d_{\hat M}(x,\pi(x)) = d_{\hat M}(x,B)$ for each $x\in\hat M$; recall that $\pi$ strictly contracts distances in $\hat M\setminus B$. This last property implies that 
\begin{equation}\label{I am sick of this}
E(\pi\circ \hat\tau|_{K}) \le E(\hat\tau|_K)
\end{equation}
with equality if and only if $\hat{\tau}(K) \subset B$. Now, the map $\tau$ is homotopic to a map $\tau'$ with $\tau'\vert_{X_G\setminus K}=\tau\vert_{X_G\setminus K}$ and such that $\tau'\vert_K$ is equal to the composition of $\pi\circ\hat\tau\vert_K$ with the projection $\hat M\to M_G$. Since $\tau$ is a minimizer of the energy we deduce that $E(\tau')\ge E(\tau)$. This implies that in \eqref{I am sick of this} we must have equality, which prove that $\hat{\tau}(K) \subset B$ and hence that $\tau(K)\subset B$.
This concludes the proof of the lemma.
\end{proof}

\subsection{Proof of Proposition \ref{prop:good extension}}
We are now ready to prove Proposition \ref{prop:good extension}. Let $\tau:X_G\to M_G$ be the map provided by Lemma \ref{lemma:BCG} and note that, up to a small perturbation, we might assume that $\tau$ is a piecewise smooth immersion. Let $\rho$ be the metric on $X_G$ obtained as the pull-back via $\tau$ of the metric of $M_G$ and let $\tau'$ be the energy minimizer in $\mathcal{F}(X_G;\tau^H)$ with respect to $\rho$ as provided by Lemma \ref{lem-dirichlet}. We get from Lemma \ref{lemma:facts_on_minimizers} that $\tau'$ satisfies the second property claimed in the statement of Proposition \ref{prop:good extension}. In other words, it suffices to prove that it satisfies the area bound. We will show that
\begin{equation}
\label{eq:Area_E_C}
\Area(\tau') \le E(\tau';\rho) \le \Area(\tau).
\end{equation}
The first inequality in \eqref{eq:Area_E_C} follows from the pointwise relation $\Jac_2(\tau') \le \norm{D\tau'}^2_\rho$ and the co-area formula:
\begin{eqnarray*}
\Area(\tau') &=& \int_{\tau'(X_G)} \vert (\tau')^{-1}(y)\vert \mathrm{d}\mathcal{H}^2(y) \\
&=& \int_{X_G}\Jac_2(\tau') \vol_\rho \le \int_{X_G} \norm{D\tau'}^2_\rho \vol_\rho = E(\tau';\rho).
\end{eqnarray*}
The second inequality follows from the observations that $\tau$ is a competitor for the energy in $\mathcal{F}(X_G;\tau^H)$ and that $\tau$ is conformal with respect to $\rho$, which means that the Jacobian and the pointwise energy agree.  Thus, by the co-area formula,
\begin{eqnarray*}
E(\tau';\rho) &\le& E(\tau;\rho) = \int_{X_G} \norm{D\tau}^2_\rho \vol_\rho \\
&=& \int_{X_G}\Jac_2(\tau) \vol_\rho
= \int_{\tau(X_G)} \vert\tau^{-1}(y)\vert \mathrm{d}\mathcal{H}^2(y) =\Area(\tau).
\end{eqnarray*}
This concludes the proof.\qed

\subsection{A lemma about area}

We finish this section by showing that, under the assumption that $H$ fills $G$, balls in the thick part of $M_G$ which are centered at $\tau(X_G)$ contain a definite amount of the image $\tau(X_G)$. 

\begin{lem}\label{lemma:thick}
With the notation as in Proposition \ref{prop:good extension}, let $\delta>0$ and $x\in X_G$ be such that $\inj_{M_G}(\tau(x))\ge \delta$ and $d(\tau(x),\tau(Y))>\delta$. If $H$ fills $G$, then $\mathcal{H}^2(\tau(X_G) \cap B(\tau(x),\delta,M_G)) \ge \delta^2/4$.
\end{lem}

In the statement, $B(\tau(x),\delta,M_G)$ is the open metric ball about $\tau(x)$ of radius $\delta$ in $M_G$. Note that the assumption that $H$ fills $G$ implies that Lemma \ref{BS-1} applies to $Y$ and $X=X_G$.

\begin{proof}
For $t\in [\delta/2,\delta]$, let $\Omega^t_x \subset \tau^{-1}(B(\tau(x),\delta,M_G))$ be the connected component containing $x$ and let $S_x^t = \tau^{-1}(\partial\Omega_x^t)$ be the boundary of $\Omega^t_x$. Note that, for almost every $t\in [\delta/2,\delta]$, $S_x^t$ is a collared locally separating graph in $X_G$. By the co-area formula we have 
\begin{eqnarray*}
\mathcal{H}^2(\tau(\Omega_x^\delta \setminus \Omega_x^{\delta/2})) 
&\ge& \int_{\delta/2}^\delta \mathcal{H}^1(\tau(\Omega_x^\delta \setminus \Omega_x^{\delta/2}) \cap \partial B_t) \mathrm{d}t \\
&\ge& \int_{\delta/2}^\delta \mathcal{H}^1(\tau(S_x^t)) \mathrm{d}t.
\end{eqnarray*}
Thus there exists a set $E\subset [\delta/2,\delta]$ of positive measure for which
\[
\mathcal{H}^1(\tau(S_x^t)) \le \frac{2}{\delta} \mathcal{H}^2(\tau(\Omega_x^\delta))
\]
for each $t\in E$. We may also assume that $S_x^t$ is a collared locally separating graph for every $t\in E$. We fix a radius $t\in E$. 

Since $\tau(\Omega_x^t)\subset B(x,t,M_G) \subset B(x,\delta,M_G)$ and $\inj_{M_G}(\tau(x))\ge \delta$, the image of $(\tau|_{\Omega_x^t})_* \colon \pi_1(\Omega_x^t)\to \pi_1(M_G)$ is trivial. Thus, by Lemma \ref{BS-1}, $S_x^t$ contains a connected subgraph $\hat S_t$ which bounds a connected subcomplex $\hat \Omega_t\subset X_G$ with connected boundary satisfying $\Omega_x^t\subset \hat \Omega_t$, $Y\cap \hat \Omega_t=\emptyset$, and $[\pi_1(\hat \Omega_t)\to\pi_1(M_G)]=1$.

Since 
\[
\diam(\tau(S_x^t)) \le \mathcal{H}^1(\tau(S_x^t)) \le \frac{2}{\delta}\CH^2(\tau(\Omega_x^\delta)),
\]
there exists a ball $B\subset M_G$ of diameter $2\CH^2(\Omega_x^\delta)/\delta$ containing $\tau(S_x^t)$. Since $\tau(\partial \hat \Omega_t)=\tau(\hat S_t) \subset \tau(S_x^t) \subset B$, we have
\[
\tau(\Omega_x^t) \subset \tau(\hat \Omega_t) \subset B
\]
by Lemma \ref{lemma:facts_on_minimizers}. Thus
\[
\frac{2}{\delta} \CH^2(\tau(\Omega_x^\delta)) = \diam B \ge \diam \tau(\Omega_x^t) \ge d(\tau(x),\tau(S_x^t)) = t \ge \frac{\delta}{2}
\]
and the claim follows. 
\end{proof}


\section{Proof of Proposition \ref{diam-bound}}\label{sec:diam}

In this section we prove Proposition \ref{diam-bound}, following an argument modeled on the proof of Thurston's bounded diameter lemma \cite{Bonahon}. Life would be much easier if all manifolds in question had injectivity radius uniformly bounded from below and we suggest the reader to consider this case at first; however, to treat the general case we need to recall a few facts about the thin-thick decomposition of hyperbolic manifolds.

The Margulis lemma \cite{Kapovich-book} asserts that there is some constant $\mu>0$, depending only on the dimension $n$, with the property that the fundamental group $\pi_1(U)$ of a connected component $U$ of the $\mu$-{\em thin part} 
$$M^{<\mu}= \{x\in M \vert \inj_{M}(x)<\mu\}$$
of any hyperbolic $n$-manifold $M$ is virtually abelian. Moreover, the closure of $U$ is compact if and only if $\pi_1(U)$ is hyperbolic. We denote by $M^{cusp<\mu}$ the union of the unbounded components of $M^{<\mu}$ and refer to it as the {\em cuspidal part} of $M$. The complements of the thin and cuspidal parts
$$M^{\ge\mu}=M\setminus M^{<\mu}\ \ \hbox{and} \ \ M^{cusp\ge\mu}=M\setminus M^{cusp<\mu}$$
are respectively the {\em thick} and {\em non-cuspidal parts} of $M$. Note that the cuspidal part if empty if $\pi_1(M)$ is purely hyperbolic. Up to diminishing $\mu$ once and forever, we can assume that any two components of the $\mu$-thin part are at least distance $1$ from each other. 

Following Bonahon \cite{Bonahon}, we define the \emph{length relative to $M^{<\mu}$ of a path} $\gamma\colon [0,1]\to M$ as the length of the part of $\gamma$ contained in the {\em thick part} $M^{\ge\mu}$. Naturally, the \emph{distance $d_{\rel M^{<\mu}}(x,y)$ relative to $M^{<\mu}$ between points $x$ and $y$ in $M$} is the minimal relative length of a path joining both points. Note that if some point of a component of $M^{<\mu}$ is at distance less than $L$ from $x$ relative to the thin part, then the whole component is at distance less than $L$ relative to the thin part. In particular, in the presence of cusps, the set of points which are distance at most $L$ relative to the thin part can be unbounded. On the other hand, if there are no cusps, then the sets $\{d_{\rel M^{<\mu}}(x,\cdot)\le L\}$ are compact. 
\medskip

We return now to the concrete situation we are interested in. Notation is always as in ($\star$), meaning that we have a fixed torsion free subgroup $\Gamma\subset\Isom_+\BH^n$ with $\dim_\CH(\Lambda_\Gamma)<1$, a fixed finitely generated subgroup $H\subset\Gamma$, a fixed, pointed, finite connected graph $(Y,*)$ and a fixed $\pi_1$-injective piecewise smooth embedding $\tau^H \colon (Y,*)\to(M_\Gamma,*)$ with $\tau^H_*(\pi_1(Y,*))=H$.

We will be interested in subgroups $G$ of the fixed ambient group $\Gamma$, which is locally free by Kulkarni's theorem. In particular, virtually abelian subgroups of any such $G$ are cyclic, implying that components $U$ of the thin part of $M_G=\BH^n/G$ are homeomorphic to $\BS^1\times\BR^{n-1}$. 
 
Up to reducing $\mu$ once more if necessary, we can assume that the image of the marking map $\tau^H \colon Y\to M_\Gamma$ is contained in the $\mu$-thick part of $M_\Gamma$. Observe that this implies that $\tau^H(Y) \subset M_G^{\ge\mu}$ for every subgroup $G\subset \Gamma$ containing $H$.

\begin{lem}\label{lem-bounded-diam1}
With notation as in ($\star$) there is a constant $C>0$ with the following property: If $H\subset G\subset\Gamma$ is a finitely generated subgroup of $\Gamma$ filled by $H$, and if $\tau \colon X_G\to M_G$ is the extension of $\tau^H$ provided by Proposition \ref{prop:good extension}, then we have
$$d_{\rel M_G^{<\mu}}(\tau(x),\tau(y))\le C$$
for any two points $x,y\in X_G$.
\end{lem}
\begin{proof}
Since $\tau\vert_Y=\tau^H$ and the latter lifts to the cover $M_H$ of $M_G$, we have
$$\diam_{M_G}(\tau(Y))\le\diam_{M_H}(\tau^H(Y)).$$
Thus the claim follows once we bound the relative distance between the images under $\tau$ of the endpoints $\eta(0)$ and $\eta(1)$ of paths $\eta\colon[0,1]\to X_G$ which stay at least at distance $\mu$ from $\tau(Y)$.

Choose a maximal collection of points $0=t_0< t_1 < \cdots < t_k\le 1$ in $[0,1]$ with $\tau(\eta(t_i))\in M_G^{\ge\mu}$ for all $i$ and with 
$$B(\tau(\eta(t_i)),\mu,M_G)\cap B(\tau(\eta(t_j)),\mu,M_G)=\emptyset\ \ \hbox{for}\ i\neq j.$$
Then each point $\tau(\eta[0,1])\cap M_G^{\ge\mu}$ is within distance $2\mu$ of one of the points $\tau(\eta(t_0)),\dots,\tau(\eta(t_k))$, and this implies that
\begin{equation}\label{eq-boundeddiam1}
d_{\rel M_G^{<\mu}}(\tau(\eta(0)),\tau(\eta(1)))\le 2\mu(k+2).
\end{equation}
On the other hand, we get from Lemma \ref{lemma:thick} that 
$$\CH^2(\tau(X_G)\cap B(\tau(\eta(t_i)),\mu,M_G))\ge\frac{\mu^2}4$$
for each $i$, meaning that
\begin{equation}\label{eq-boundeddiam2}
\Area(\tau)\ge\sum_{i=0}^k\CH^2(\tau(X_G)\cap B(\tau(\eta(t_i)),\mu,M_G))\ge\frac{\mu^2}4(k+1).
\end{equation}
Taking together \eqref{eq-boundeddiam1} and \eqref{eq-boundeddiam2} and recalling that, since $\tau$ is provided by Proposition \ref{prop:good extension}, $\Area(\tau)$ is bounded by a constant $C_0$ independent of $G$, we get that
$$d_{M_G}(\tau(\eta(0)),\tau(\eta(1)))\le 2\mu+\frac 8\mu C_0.$$
The claim follows.
\end{proof}

We are now ready to prove Proposition \ref{diam-bound}.

\begin{named}{Proposition \ref{diam-bound}}
Let $\Gamma\subset\Isom_+(\BH^n)$ be as in Theorem \ref{main-dim}, $H\subset\Gamma$ a finitely generated non-elementary subgroup and $*\in\BH^n$ a base point. Then there is a constant $d>0$ such that each finitely generated intermediate subgroup $H\subset G\subset\Gamma$ filled by $H$ is generated by elements $g$ satisfying $d_{\BH^n}(*,g *)\le d$.
\end{named}
\begin{proof}
Suppose for the time being that $\Gamma$ is purely hyperbolic. The case with parabolics will be discussed below.

Let $C>0$ be the constant provided by Lemma \ref{lem-bounded-diam1}, suppose that $G\subset\Gamma$ is a subgroup filled by $H$, let $\tau\colon X_G\to M_G$ be the extension of $\tau^H$ provided by Proposition \ref{prop:good extension}, and let $x\in X_G$ be an arbitrary point. By Lemma \ref{lem-bounded-diam1}, the point $\tau(x)$ belongs to the set $\{z\in M_G\vert d_{\rel M_G^{<\mu}}(z,\tau(*))\le C\}$ of points at relative distance at most $C$ from the point $\tau(*)=*$. This means in particular that there is a path $\eta\colon [0,1]\to M_G$ with $\eta(0)=*$, with $\eta(1)=\tau(x)$ and which has at most length $C$ relative to the thin part $M_G^{<\mu}$. 

Since any two components of the $\mu$-thin part of $M_G$ are at distance at least $1$ from each other, it follows that $\eta$ enters at most $C+1$ distinct components of $M_G^{<\mu}$. This means that the actual distance in $M_G$ between the endpoints $\eta(0)=\tau(*)=*$ and $\eta(1)=\tau(x)$ is bounded from above in terms of $C$ and the diameter of the components of the thin part of $M_G$ traversed by the path $\eta$. 

Each one of the components of the thin part traversed by $\eta$ projects under the cover $M_G\to M_\Gamma$ into a component of the thin part of $M_\Gamma$ contained in the compact set $\{z\in M_\Gamma\vert d_{\rel M_\Gamma^{<\mu}}(z,*)\le C\}$ -- compactness being a consequence of the assumption that $\Gamma$ is purely hyperbolic. Compactness of $\{z\in M_\Gamma\vert d_{\rel M_\Gamma^{<\mu}}(z,*)\le C\}$ implies that there is a positive lower bound for the length of closed geodesics contained therein. This implies that the lengths of the closed geodesics in $M_G$ contained in $\{z\in M_G\vert d_{\rel M_G^{<\mu}}(z,\tau(*))\le C\}$ are bounded from below by a positive constant independent of $G$. In turn, this implies that there is an upper bound independent of $G$ for the diameters of the components of the thin part $M_G^{<\mu}$ traversed by $\eta$. 

Taking all this together, we have proved that the distance between $\tau(x)$ and $*$ is bounded independently of $G$. We record this fact:

\begin{fact}\label{fact1}
Suppose that $\Gamma$ is purely hyperbolic. Then there is a constant $D>0$ such that for each intermediate group $H\subset G\subset\Gamma$ filled by $H$ one has
$$d_{M_G}(\tau(x),*)\le D$$
for all $x\in X_G$. Here $\tau:X_G\to M_G$ is the extension of $\tau^H$ provided by Proposition \ref{prop:good extension}. \qed
\end{fact}

Armed with this fact we can conclude the proof of Proposition \ref{diam-bound} in the case that $\Gamma$ is purely hyperbolic. We  use an argument similar to that in the proof of Proposition 5.28 in \cite{Gromov-book}.

Given any element $g\in G$ consider a path $\gamma\colon [0,1]\to\tilde X_G$ connecting $*$ and $g(*)$; here $\tilde X_G$ is the universal cover of $X_G$ and $G=\pi_1(X_G,*)$ acts on $\tilde X_G$ by deck-transformations. Consider the image of $\gamma$ under the lift $\tilde\tau\colon \tilde X_G\to\BH^n$ of the extension $\tau:X_G \to M_G$ and choose points $t_0=0<t_1<\dots<t_k=1$ such that
$$d_{M_G}(\tilde\tau(\gamma(t_i)),\tilde\tau(\gamma(t_{i+1})))\le 1$$
for all $i$. By Fact \ref{fact1}, each one of the points $\tilde\tau(\gamma(t_i))$ can be joined to some element of the orbit $G*\subset\BH^n$ of the base point $*$ by a path $\eta_i$ of length at most $D$. Now, setting
$$\sigma_i=\tilde\tau\circ\eta\vert_{[t_{i-1},t_i]}$$
we have that the path $\tilde\tau\circ\gamma$ is the concatenation of paths $\sigma_1,\dots,\sigma_k$ and hence homotopic relative to its endpoints to the concatenation of $\sigma_1$, $\eta_1$, $\eta_1^{-1}$, $\sigma_2$, $\eta_2$, $\eta_2^{-1}$, $\sigma_3$, $\dots$, $\sigma_{k-1}$, $\eta_{k-1}$, $\eta_{k-1}^{-1}$ and $\sigma_k$. We have thus the following factorization of $g$ in $G=\pi_1(M_G,*)$: 
$$g=(\sigma_1\eta_1)\cdot(\eta_1^{-1}\sigma_2\eta_2)\cdot(\eta_2^{-1}\sigma_3\eta_3)\cdot\dots\cdot(\eta_{k-2}^{-1}\sigma_{k-1}\eta_{k-1})\cdot(\eta_{k-1}^{-1}\sigma_k).$$
By construction each one of these elements in represented by a loop of length at most $d=2D+1$. Since $g$ was arbitrary, this concludes the proof of Proposition \ref{diam-bound} in the case that $\Gamma$ is purely hyperbolic.
\medskip

We discuss now the modifications needed to deal with the case that $\Gamma$ has parabolic elements. First we note that the proof of Fact \ref{fact1} still applies as long as we do not allow ourselves to travel through the cuspidal part of $M_G$. More concretely we have:

\begin{fact}\label{fact2}
There is a constant $D>0$ such that for each intermediate group $H\subset G\subset\Gamma$ filled by $H$ one has
$$d_{M_G}(\tau(x),*)\le D$$
for every point $x\in X_G$ for which there is a path $\eta:[0,1]\to X_G$ with $\eta(0)=*$, $\eta(1)=x$ and $\tau(\eta([0,1]))\subset M_G^{cusp\ge\mu}$. Here $\tau:X_G\to M_G$ is the extension of $\tau^H$ provided by Proposition \ref{prop:good extension}. \qed
\end{fact}

Our next goal is to see that $\tau^{-1}(M_G^{cusp\ge\mu})$ is connected. Taking intersections of $\tau(X_G)$ with the boundary of the cuspidal components of $M_G^{<\mu}$ we obtain from the Seifert-van Kampen theorem a description of $G=\tau_*(\pi_1(X_G))$ as the fundamental group of an essential graph of groups $\CG$ whose edge groups are of the form $\tau_*(\pi_1(Z))$ with $Z$ connected component of $\tau^{-1}(\D M_G^{cusp\ge\mu})$ and whose vertex groups are of the form $\tau_*(\pi_1(V))$ where $V\subset X_G$ are connected components of either $\tau^{-1}(M_G^{cusp<\mu})$ or $\tau^{-1}(M_G^{cusp\ge\mu})$. Note that the edge groups as well as the vertex groups corresponding to components of $\tau^{-1}(M_G^{cusp<\mu})$ in $\CG$ are parabolic subgroups of $G$. On the other hand, the convexity properties of $\tau$ (c.f.\;Proposition \ref{prop:good extension}) imply that the vertex groups corresponding to components $\tau^{-1}(M_G^{cusp\ge \mu})$ are non-elementary. We leave the details to the reader.

Proving that $\tau^{-1}(M_G^{cusp\ge\mu})$ is connected, amounts to proving that $\CG$ has a single vertex of the form $\tau_*(\pi_1(V))$ with $V$ a connected component of $\tau^{-1}(M_G^{cusp\ge\mu})\subset X_G$:

\begin{fact}\label{fact3}
The graph of groups $\CG$ has only a single vertex of the form $\tau_*(\pi_1(V))$ with $V$ connected component of $\tau^{-1}(M_G^{cusp\ge\mu})$.
\end{fact}
\begin{proof}

Suppose $\tau_*(\pi_1(V))$ and $\tau_*(\pi_1(W))$ are vertex groups of $\CG$ for two different components $V$ and $W$ of $\tau^{-1}(M_G^{cusp \ge \mu})$, and suppose for the sake of concreteness that $V\subset X_G$ is the component containing $Y$. The vertex groups $\tau_*(\pi_1(V))$ and $\tau_*(\pi_1(W))$ are finitely generated, non-elementary, and -- being subgroups of $\Gamma$ -- their limit sets are Cantor sets.  As remarked after the statement of Kulkarni's theorem in the introduction, $\tau_*(\pi_1(W))$ splits as a non-trivial free product relative to its parabolic subgroups. Any such splitting induces a free splitting of $\pi_1(\CG)$ relative to all its vertex groups but $\tau_*(\pi_1(W))$, and in particular relative to $\tau_*(\pi_1(V))$. Since $H=\tau_*(\pi_1(Y))\subset\tau_*(\pi_1(V))$, this contradicts the assumption that $H$ fills $G$.
\end{proof}

Let from now on $\tau_*(\pi_1(V))$ be the unique vertex group of $\CG$ corresponding to a component $V$ of $\tau^{-1}(M_G^{cusp\ge\mu})$, note that $\tau(V)$ has diameter (as a subset of $M_G$) bounded from above independently of $G$ by Fact \ref{fact2}. The same argument as in the proof in the purely hyperbolic case shows that the vertex group $\hat H=\tau_*(\pi_1(V))$ is generated by elements of length bounded by some constant which does not depend on $G$. In particular, when $G$ varies, $\hat H$ may also change, but only within a finite set of subgroups of $\Gamma$. Note that, since (1) $G=\pi_1(\CG)$ is obtained from $\hat H$ by amalgamating over parabolic subgroups, (2) $\hat H$ has finitely many conjugacy classes of parabolic subgroups, (3) each such subgroup has finite index in the corresponding maximal parabolic subgroup in $\Gamma$, and (4) maximal parabolic subgroups are self-normalizing, there are only finitely many choices for the group $G$ for any given $\hat H$. Thus there are only finitely many choices for $G$ altogether. The uniform upper bound for the lengths of generators of $G$ follows -- we leave the details to the reader. This concludes the proof of Proposition \ref{diam-bound}.
\end{proof}


\section{Locally free group having a relatively small limit set}\label{sec:grope}

In this section we prove Theorem \ref{grope}:

\begin{named}{Theorem \ref{grope}}
For each $\epsilon>0$ there is a discrete, purely hyperbolic subgroup $\Gamma$ of $\Isom_+(\BH^3)$ which is not free but whose limit set $\Lambda_\Gamma$ is a Cantor set with Hausdorff dimension $\dim_\CH(\Lambda_\Gamma)\le 1+\epsilon$.
\end{named}

Let $r_0>0$ be a large constant to be fixed later, $r>r_0$, and let $S_r$ a compact hyperbolic surface for which 
\begin{itemize}
\item[(a)] $\D S_r$ consists of a single geodesic of length $1$ and has a tubular neighborhood of width at least $r$, and 
\item[(b)] $S_r$ contains a non-separating simple closed geodesic $\gamma$ of length $1$ and which again has a tubular neighborhood of width at least $r$.
\end{itemize}


\begin{figure}[h!]
\begin{overpic}[scale=0.2,unit=1mm]{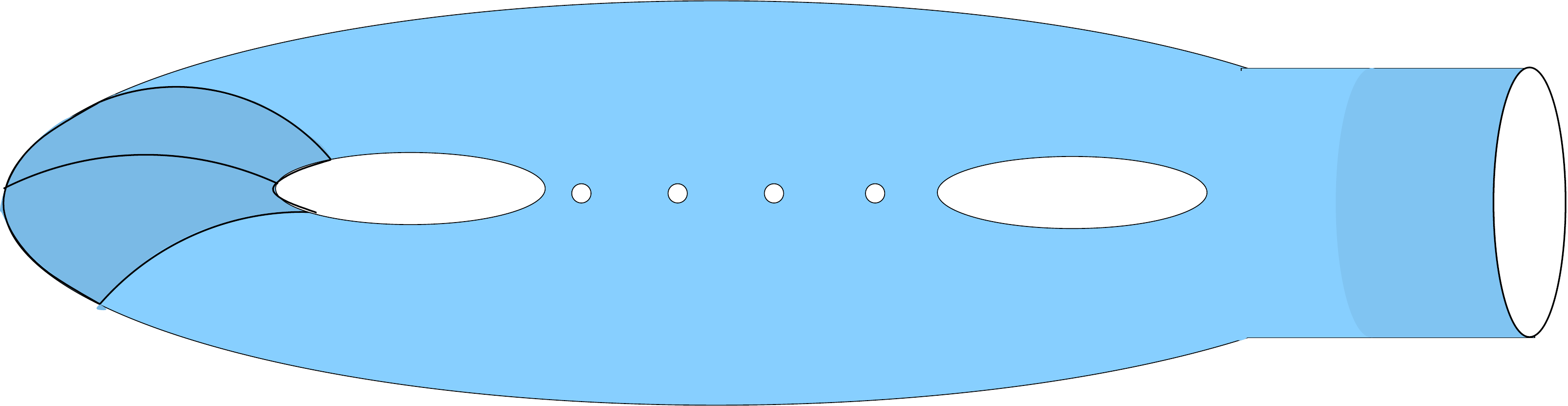}
\put(-2,10){\tiny $\gamma$}
\end{overpic}
\caption{Surface $S_r$} 
\label{fig:Surface_Sr}
\end{figure}

The fundamental group of $S_r$ is free and we choose a standard free basis 
\[
\pi_1(S_r,*)=\langle\gamma_1,\dots,\gamma_{2g}\vert -\rangle
\]
with $\gamma_1=\gamma$, where $*\in\gamma$ is a fixed base point. Note that $\D S_r$ is (freely) homotopic to $\prod_{i=1}^g[\gamma_{2i-1},\gamma_{2i}]$, and that for $r$ large the area of $S_r$ is also large, meaning that the genus $g$ is also large. 

Let now $Y_r$ be the 2-complex obtained from $S_r$ by gluing isometrically $\D S_r$ onto the geodesic $\gamma$, and endow $Y_r$ with the largest path metric with respect to which the quotient map $S_r\to Y_r$ preserves lengths of paths. The fundamental group $\pi_1(Y_r)$ of $Y_r$ is an HNN-extension of $\pi_1(S_r)$, i.e.\;it has a presentation
\begin{equation}\label{eq-presentation}
\pi_1(Y_r,*)=\left\langle\gamma_1,\dots,\gamma_{2g},\tau\;\Big\vert\;\gamma_1=\tau\left(\prod_{i=1}^g[\gamma_{2i-1},\gamma_{2i}]\right)\tau^{-1}\right\rangle,
\end{equation}
where $\tau$ is the stable letter and the base point $*\in Y_r$ is the image of the base point of $S_r$ under the quotient map $S_r\to Y_r$. 

The universal cover $\tilde Y_r$ of $Y_r$ consists of isometrically embedded copies of the universal cover $\tilde S_r$ of $S_r$ glued along geodesic curves; we refer to these copies as {\em strata}. Choose also a lift $\tilde *$ of the base point $*\in Y_r$ and let $\tilde S_r^0$ be the stratum which contains $\tilde *$ in its interior. In other words, $\tilde S_r^0$ is the image of the map $\tilde S_r\to\tilde Y_r$ which covers the quotient map $S_r\to Y_r$ and maps a lift of $*\in S_r$ to the lift $\tilde *\in\tilde Y_r$.

The universal cover $\tilde S_r$ admits a unique (up to isometry) embedding into the hyperbolic plane $\BH^2$. Choose a copy of $\BH^2$ in $\BH^3$ and note that the induced isometric embedding 
\[
\tilde\phi_r \colon \tilde S_r^0\to\BH^2\subset\BH^3
\]
is equivariant under some discrete Fuchsian representation 
\[
\rho^0_r \colon \pi_1(S_r)\to\Isom_+(\BH^2)\subset\Isom_+(\BH^3).
\]
We extend $\rho^0_r$ to a representation 
\[
\rho_r \colon \pi_1(Y_r)\to\Isom_+(\BH^3)
\]
by choosing $\rho_r(\tau)$ to be one of the elements in $\Isom_+(\BH^3)$ mapping the axis of $\rho_r^0(\gamma_1)$ to the axis of $\rho_r^0(\prod_{i=1}^g[\gamma_{2i-1},\gamma_{2i}])$ in such a way that $\tau(\BH^2)$ meets $\BH^2$ orthogonally; here $\tau$ is as in \eqref{eq-presentation}.

The isometric embedding $\tilde\phi_r \colon \tilde S_r^0\to\BH^3$ extends now uniquely to a $\rho_r$-equivariant map
\[
\tilde\phi_r\colon \tilde Y_r\to\BH^3,
\]
which preserves lengths of paths and hence is $1$-Lipschitz. More concretely, the image under $\tilde\phi_r$ of a geodesic in $\tilde Y_r$ is a piecewise geodesic path consisting of geodesics segments of length at least $2r$ and meeting with angles at least $90^\circ$. Thus, given $\epsilon>0$ small, there are constants $r_0>0$ and $C>0$ so that for $r\ge r_0$ the image of a geodesic $t\mapsto\eta(t)$ in $\tilde Y_r$ is a $(1+\epsilon,C)$-quasi-geodesic in $\BH^3$, that is, 
\[
\frac {1}{1+\varepsilon}d_{\tilde Y_r}(\eta(s),\eta(t))-C\le d_{\BH^3}(\tilde\phi_r(\eta(s)),\tilde\phi_r(\eta(t)))\le (1+\epsilon)d_{\tilde Y_r}(\eta(s),\eta(t))+C.
\]
for $t,s\in \BR$. We summarize these observations as follows:

\begin{lem}
For each $\epsilon>0$ there are constants $r_0>0$ and $C>0$ such that the $\rho_r$-equivariant map $\tilde\phi_r\colon \tilde Y_r\to\BH^3$ is a $(1+\epsilon,C)$-quasi-isometry for all $r\ge r_0$.\qed
\end{lem}

This lemma implies in particular that the representation $\rho_r$ is discrete and faithful with convex-cocompact image for $r\ge r_0$. In particular, $\rho_r(\pi_1(Y_r))$ is purely hyperbolic. We show next that its limit set $\Lambda_{\rho_r}=\Lambda_{\rho_r(\pi_1(Y_r))}$ has relatively small Hausdorff dimension $\dim_{\CH}(\Lambda_{\rho_r})$.

Since the representation $\rho_r$ is convex-cocompact, the Hausdorff dimension of $\Lambda_{\rho_r}$ equals the critical exponent $\delta_{\rho_r}$ of the image of $\rho_r$. Since the map $\tilde\phi_r$ is a $(1+\epsilon,C)$-quasi-isometry, we can bound the critical exponent of $\rho_r(\pi_1(Y_r))$ by the critical exponent of the action $\pi_1(Y_r)\actson\tilde Y_r$ times $1+\epsilon$. Since the action $\pi_1(Y_r)\actson \tilde Y_r$ is cocompact, its critical exponent coincides with the volume growth entropy 
\[
h(Y_r):=\limsup_{R\to\infty}\frac{\ln \vol(B(z,R,\tilde Y_r))}{R}
\]
of $\tilde Y_r$. Here $B(z,R,\tilde Y_r)$ is the ball of radius $R$ centered at some fixed by otherwise arbitrary point $z\in\tilde Y_r$ and $\vol(B(z,R,\tilde Y_r))$ is its area; for the sake of concreteness suppose that $z$ is a manifold point contained in the stratum $\tilde S_r^0$.

By combining these observations, we have
\begin{equation}\label{eq-dim-vol}
\dim_\CH(\Lambda_{\rho_r})\le (1+\epsilon) h(Y_r)
\end{equation}
for all $r\ge r_0$ and it suffices to estimate the volume growth entropy $h(Y_r)$ from above.

The embedding $\tilde \phi_r \colon \tilde S_r^0\to \BH^2$ extends to a map
\[
\pi\colon \tilde Y_r\to\BH^2
\]
with the following properties\footnote{One could compare the map $\pi$ to the projection of a building to an apartment.}: 
\begin{itemize}
\item[(a)] $\pi$ preserves the lengths of paths and maps geodesic rays based at $z$ to geodesic rays in $\BH^2$, and
\item[(b)] the restriction of $\pi$ to each stratum of $\tilde Y_r$ is an isometric embedding.
\end{itemize}

The preimage under $\pi$ of a geodesic segment $[\pi(z),a]$ in $\BH^2$ is a (possibly degenerate) tree with trivalent branching corresponding to the branching of $\tilde Y_r$. Points in the preimage $\pi^{-1}(a)$ of $a$ are leaves of the tree $\pi^{-1}([\pi(z),a])$. Since the distance along the tree between any two branching points is at least $2r$, we can the bound (generously) the number of leaves, and hence the cardinality of $\pi^{-1}(a)$ by 
\begin{equation}\label{eq-pobrepepe}
\vert\pi^{-1}(a)\vert\le 2^{1+\frac{d(\pi(z),a)}{2r}}.
\end{equation}
Taking into account that 
\[
\pi(B(z,R,\tilde Y_r))=B(\pi(z),R,\BH^2)\ \ \hbox{and}\ \ \pi^{-1}(B(\pi(z),R,\BH^2))=B(z,R,\tilde Y_r)
\]
we obtain
\[
\vol(B(z,R,\tilde Y_r))=\int_{B(\pi(z),R,\BH^2)}\vert\pi^{-1}(a)\vert \vol_{\BH^2}(a) 
\le 2\pi\int_0^R 2^{1+\frac{t}{2r}}\sinh(t)\mathrm{d}t.
\]
This implies immediately that 
\begin{equation}\label{eq-vol}
h(Y_r)\le 1+\frac{\ln 2}{2r}.
\end{equation}
Taking together \eqref{eq-dim-vol} and \eqref{eq-vol} we deduce:

\begin{lem}\label{lem-dim-bound}
For each $\epsilon>0$ there is $r_0>0$ with
\[
\dim_\CH(\Lambda_{\rho_r(\pi_1(Y_r))})\le 1+\epsilon
\]
for all $r\ge r_0$.\qed
\end{lem}

The limit set of the group $\rho_r(\pi_1(Y_r))$ has small Hausdorff dimension but is connected, and hence does not, by itself, prove Theorem \ref{grope}. We find our desired group $\Gamma$ as a subgroup of $\rho_r(\pi_1(Y_r))$. Starting with the construction of our subgroup, we note that, by the presentation \eqref{eq-presentation} of $\pi_1(Y_r)$, there exists a homomorphism $\sigma \colon \pi_1(Y_r) \to \BZ$ satisfying $\tau \mapsto 1$ and $\gamma_i \mapsto 0$ for each $i=1,\ldots, 2g$.

Let $\tilde Y_r$ be the universal cover of $Y_r$ and $\hat Y_r=\tilde Y_r/\Ker(\sigma)$ the cover corresponding to the kernel of $\sigma$. By construction, the group $\BZ$ acts on $\hat Y_r$ with $\hat Y_r/\BZ=Y_r$. In fact, the quotient map $S_r\to Y_r$ lifts to an embedding $S_r\hookrightarrow\hat Y_r$ whose image we denote by $S_r^0$. For $k\in\BZ$, let $S_r^k$ be the translate of $S_r^0$ under the action $\BZ\actson\hat Y_r$. Then
\[
\hat Y_r=\cup_{k\in\BZ} S_r^k
\]
and any two consecutive ones $S_r^k$ and $S_r^{k+1}$ intersect on the curve corresponding to the boundary of $S_r^{k+1}$, that is, 
\[
S_r^{k+1}\cap S_r^k=\D S_r^{k+1}.
\]
We consider the complex 
\[
X_r=\cup_{k\ge 0}S_r^k
\]
consisting of the positive orbit of $S_r$ and let $\Gamma_r=\rho_r(\pi_1(X_r))$ be the image of its fundamental group under the representation $\rho_r$.


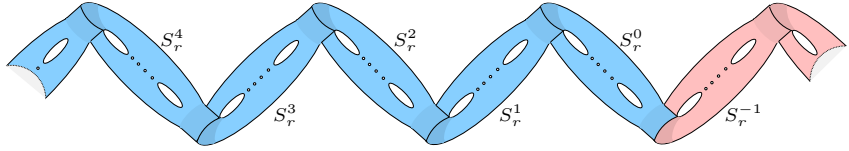
\begin{figure}[h!]
\begin{overpic}[scale=0.08,unit=1mm]{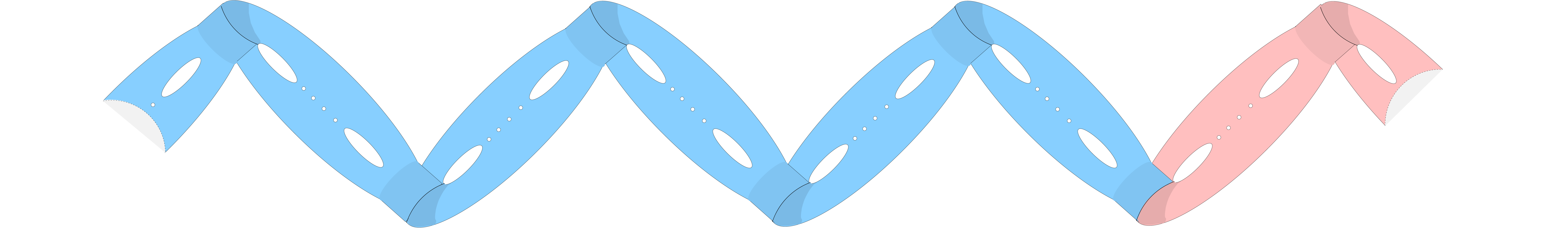}
\put(29,15){\tiny $S^4_r$} \put(44,5){\tiny $S^3_r$} \put(60,15){\tiny $S^2_r$} 
\put(74,5){\tiny $S^1_r$} \put(90,15){\tiny $S^0_r$} \put(104,5){\tiny $S^{-1}_r$}
\end{overpic}
\caption{The branching surface $\hat Y_r$ near $S^0_r$; $X_r$ in blue.}
\label{fig:Surface_Xr}
\end{figure}

The universal cover $\tilde X_r$ of $X_r$ is convex subset of the locally compact CAT(-1) space $\tilde Y_r$, and hence a locally compact CAT(-1) space itself. To understand the structure of its Gromov boundary $\D_\infty\tilde X_r$, consider two points $\theta,\eta\in\D_\infty\tilde X_r$ and a geodesic $g \colon \BR\to\tilde X_r$ with $\theta=g(\infty)$ and $\eta=g(-\infty)$. Suppose for the sake of concreteness that $g(\BR)$ intersects the manifold part of the base stratum $\tilde S_r^0$ in a compact set. This means that our geodesic $g(\BR)$ enters $\tilde S_r^0$ through a lift $\tilde\gamma_1$ of $\gamma_1$ and leaves through another one $\tilde\gamma_1'$. Both geodesics $\tilde\gamma_1$ and $\tilde\gamma_1'$ separate (individually) the space $\tilde X_r$ and can be separated by a compact set $K\subset\tilde S_r^0$; see Figure \ref{fig:On_lowest_stratum}. This proves that the limit points $\theta=g(\infty)$ and $\eta=g(-\infty)$ belong to different connected components of the boundary $\D_\infty\tilde X_r$ of $\tilde X_r$. A similar argument applies to all geodesics in $\tilde X_r$, meaning that the Gromov boundary $\D_\infty\tilde X_r$ is totally disconnected. Since $\D_\infty\tilde X_r$ is also perfect, it is a Cantor set. We leave the details to the reader.


\begin{figure}[h!]
\begin{overpic}[scale=.25,unit=1mm]{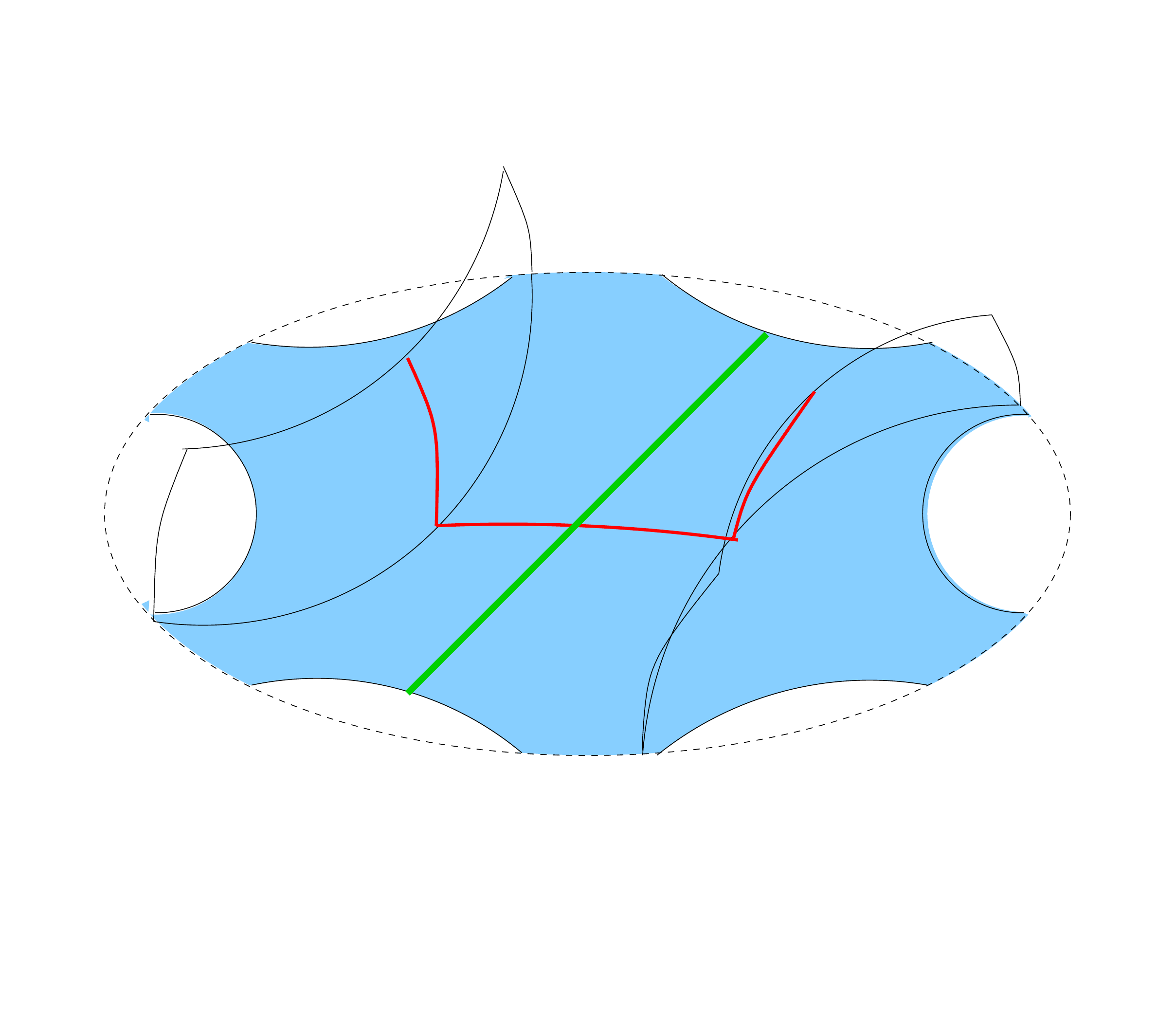}
\put(20,30){\tiny $g$}\put(33,33){\tiny $K$}
\end{overpic}
\caption{The geodesic $g$ (in red) entering and leaving the stratum $\tilde S^0_r$ (in blue). The segment $K$ is an example of a desired compact set.} 
\label{fig:On_lowest_stratum}
\end{figure}

Returning now to our previous setting, note that the boundary map $\D_\infty\tilde\phi$ of the $\rho$-equivariant quasi-isometry $\tilde\phi$ maps the Cantor set $\D_\infty\tilde X_r$ into some $\Gamma_r$-invariant Cantor set $\D_\infty \tilde\phi(\D_\infty\tilde X_r)\subset \D_\infty\BH^3$. We deduce that $\D_\infty\tilde\phi(\D_\infty\tilde X_r)$ contains the limit set $\Lambda_{\Gamma_r}$ of $\Gamma_r=\rho(\pi_1(X_r))$. 

\begin{lem}\label{lem-cantor}
The limit set $\Lambda_{\Gamma_r}$ is a Cantor set.\qed
\end{lem}

Since, by construction, $\Gamma_r$ is a subgroup of the convex-cocompact subgroup $\rho_r(\pi_1(Y_r))$, it is purely hyperbolic. In particular, we obtain from Lemma \ref{lem-cantor} and Kulkarni's theorem that $\Gamma_r$ is locally free. We claim that it is not free.  

\begin{lem}\label{lem-not-free}
The group $\Gamma_r$ is not free.
\end{lem}
\begin{proof}
We claim that, if $\Gamma_r=A*A'$ is a free splitting with $\D S_r^0\in A$, then $\Gamma_r=A$. Note first that, since $\partial S_r^0$ is connected, there is no free splitting of $\pi_1(S_r^0)$ relative to the boundary. Thus $\pi_1(S_r^0)\subset A$. Furthermore, since $\D S_r^1\in\pi_1(S_r^0)\subset A$, the same argument shows also that $\pi_1(S_r^1)\subset A$, that is, 
\[
\pi_1(S_r^0\cup S_r^1)=\pi_1(S_r^0)*_{\pi_1(\D S_r^1)}\pi_1(S_r^1)\subset A.
\]
Now, iterating the same argument, we obtain 
\[
\pi_1(\cup_{k=0}^mS_r^k)\subset A
\]
for all $m\ge 1$. Thus $\Gamma_r\subset A$, as we needed to prove.
\end{proof}

At this point we have all ingredients needed to finalize the proof of Theorem \ref{grope}.

\begin{proof}[Proof of Theorem \ref{grope}]
For given $\epsilon>0$, let $r_0>0$ be as in Lemma \ref{lem-dim-bound} and fix $r\ge r_0$. Then, by the said lemma, the limit set of the convex cocompact group $\rho_r(\pi_1(Y_r))$ has Hausdorff dimension less than $1+\epsilon$. In particular, the subgroup $\Gamma_r\subset\rho(\pi_1(Y_r))$ is discrete, purely hyperbolic and satisfies
\[
\dim_\CH(\Lambda_{\Gamma_r})\le 1+\epsilon.
\]
On the other hand, $\Lambda_{\Gamma_r}$ is a Cantor set by Lemma \ref{lem-cantor} and $\Gamma_r$ is not free by Lemma \ref{lem-not-free}. 
\end{proof}

The construction we just gave is very flexible. For example, the surfaces $S_r$ can be chosen so that the injectivity radius is bounded from below by $1$, which implies that the injectivity radius of $\BH^3/\Gamma_r$ is also uniformly bounded from below by $1$ for large $r$. We could have also chosen to the fix the topological type of the surfaces $S_r$ by considering hyperbolic metrics on them for which the boundary $\partial S_r$ and curve $\gamma$ become short when $r$ is large. Note that in this case the groups $\Gamma_r$ would have been isomorphic to each other. With small modifications one could also obtain discrete torsion free perfect subgroups $\Gamma$ of $\Isom_+(\BH^n)$ whose limit sets are Cantor sets of Hausdorff dimension less than $2$; recall that a group is \emph{perfect} if its abelianization is trivial. However, we do not know how to construct such a group $\Gamma$ with $\Lambda_\Gamma$ of Hausdorff dimension close to $1$.

\begin{quest*}
Is there $\epsilon>0$ such that every perfect torsion free discrete subgroup of $\Isom_+(\BH^n)$ satisfies $\dim_\CH(\Lambda_\Gamma)\ge 1+\epsilon$?
\end{quest*}

The answer is positive for $\Gamma\subset\Isom_+(\BH^3)$ finitely generated, since any such group is either cocompact or has positive first Betti number, and in the former case one has that its limit set $\Lambda_\Gamma$ is the whole boundary at infinity $\D_\infty\BH^3=\BS^2$, meaning that it has Hausdorff dimension $2$.
\medskip

\bibliographystyle{abbrv}
\bibliography{locallyfree}

\begin{thebibliography}{10}

\bibitem{BCG}
G.~Besson, G.~Courtois, and S.~Gallot.
\newblock Rigidity of amalgamated products in negative curvature.
\newblock {\em J. Differential Geom.}, 79(3):335--387, 2008.

\bibitem{Bonahon}
F.~Bonahon.
\newblock Bouts des vari\'et\'es hyperboliques de dimension {$3$}.
\newblock {\em Ann. of Math. (2)}, 124(1):71--158, 1986.

\bibitem{Eells-Fuglede-book}
J.~Eells and B.~Fuglede.
\newblock {\em Harmonic maps between {R}iemannian polyhedra}, volume 142 of
  {\em Cambridge Tracts in Mathematics}.
\newblock Cambridge University Press, Cambridge, 2001.
\newblock With a preface by M. Gromov.

\bibitem{Gromov-book}
M.~Gromov.
\newblock {\em Metric structures for {R}iemannian and non-{R}iemannian spaces},
  volume 152 of {\em Progress in Mathematics}.
\newblock Birkh{\"a}user Boston Inc., Boston, MA, 1999.

\bibitem{Kapovich-book}
M.~Kapovich.
\newblock {\em Hyperbolic manifolds and discrete groups}, volume 183 of {\em
  Progress in Mathematics}.
\newblock Birkh\"auser Boston, Inc., Boston, MA, 2001.

\bibitem{KK}
M.~Kapovich.
\newblock Kleinian groups in higher dimensions.
\newblock In {\em Geometry and dynamics of groups and spaces}, volume 265 of
  {\em Progr. Math.}, pages 487--564. Birkh\"auser, Basel, 2008.

\bibitem{Kapovich-GAFA}
M.~Kapovich.
\newblock Homological dimension and critical exponent of {K}leinian groups.
\newblock {\em Geom. Funct. Anal.}, 18(6):2017--2054, 2009.

\bibitem{Kulkarni}
R.~S. Kulkarni.
\newblock Infinite regular coverings.
\newblock {\em Duke Math. J.}, 45(4):781--796, 1978.

\bibitem{McMullen}
C.~T. McMullen.
\newblock Hausdorff dimension and conformal dynamics. {I}. {S}trong convergence
  of {K}leinian groups.
\newblock {\em J. Differential Geom.}, 51(3):471--515, 1999.

\bibitem{Ratcliffe}
J.~G. Ratcliffe.
\newblock {\em Foundations of hyperbolic manifolds}, volume 149 of {\em
  Graduate Texts in Mathematics}.
\newblock Springer, New York, second edition, 2006.

\bibitem{Serre-book}
J.-P. Serre.
\newblock {\em Trees}.
\newblock Springer Monographs in Mathematics. Springer-Verlag, Berlin, 2003.
\newblock Translated from the French original by John Stillwell, Corrected 2nd
  printing of the 1980 English translation.

\bibitem{SoutoNM}
J.~Souto.
\newblock {Two applications of the natural map}.
\newblock {\em Geom. Dedicata}, 133:51--57, 2008.

\bibitem{Sullivan}
D.~Sullivan.
\newblock The density at infinity of a discrete group of hyperbolic motions.
\newblock {\em Inst. Hautes \'Etudes Sci. Publ. Math.}, (50):171--202, 1979.

\bibitem{Sullivan84}
D.~Sullivan.
\newblock Entropy, {H}ausdorff measures old and new, and limit sets of
  geometrically finite {K}leinian groups.
\newblock {\em Acta Math.}, 153(3-4):259--277, 1984.

\end{thebibliography}

\end{document}